\documentclass[reqno,final,11pt]{amsart}

\usepackage[utf8]{inputenc}
\usepackage[T1]{fontenc}
\usepackage{amsmath,amssymb,amsfonts,amsthm}
\usepackage{geometry}
\usepackage{enumitem}
\usepackage{xcolor}
\usepackage{hyperref}
\usepackage{natbib}
\usepackage{url}
\usepackage{lineno}
\usepackage{graphicx}

\definecolor{linkcolor}{rgb}{0.16,0.42,0.70}
\hypersetup{
  colorlinks=true,
  linkcolor=linkcolor,
  citecolor=linkcolor,
  urlcolor=linkcolor
}

\newcommand{\Z}{\mathbb Z}

\newcommand{\Nzero}{\mathbb N_0}
\newcommand{\E}{\mathbb E}

\newcommand{\Pbb}{\mathbb P}

\newcommand{\dd}{\,d}

\theoremstyle{plain}
\newtheorem{theorem}{Theorem}[section]
\newtheorem{proposition}[theorem]{Proposition}
\newtheorem{lemma}[theorem]{Lemma}
\newtheorem{corollary}[theorem]{Corollary}

\theoremstyle{definition}
\newtheorem{definition}[theorem]{Definition}

\theoremstyle{remark}
\newtheorem{remark}[theorem]{Remark}

\numberwithin{equation}{section}

\title[The frog model with random Weibull lifetimes]
{The Frog Model on \(\mathbb Z\) with Random Discrete Weibull Lifetimes and Biased Nearest-Neighbour Random Walks}

\author[J.H. Ram\'irez-Gonz\'alez]{J.H. Ram\'irez-Gonz\'alez}
\address{Universidade de S\~ao Paulo, Brazil}
\email{hermenegildo.ramirez@usp.br}

\author[F.P. Machado]{Fabio Prates Machado}
\address{Universidade de S\~ao Paulo, Brazil}
\email{fmachado@ime.usp.br}

\renewcommand{\subjclassname}{2020 Mathematics Subject Classification}



\begin{document}
%\linenumbers
\begin{abstract}
We study the frog model on \(\mathbb Z\) with particle-wise random discrete
Weibull lifetimes and biased nearest-neighbour random walks. Each particle has
an independent survival parameter \(\pi\in(0,1)\). Conditionally on \(\pi=p\),
its lifetime \(\Xi\) satisfies
\[
  \Pbb(\Xi\ge k\mid \pi=p)=p^{k^\gamma},
  \qquad k\in\Nzero,
\]
where \(\gamma>0\). The distribution of \(\pi\) is assumed to have right-edge
density
\[
  f_\pi(u)\sim (1-u)^{\beta-1}
  L\left(\frac1{1-u}\right),
  \qquad u\uparrow1,
\]
where \(\beta>0\) and \(L:(0,\infty)\to(0,\infty)\) is a slowly varying function at infinity. The main step is to
estimate the tail of the maximal displacement of a single particle before
 death. If \(\tau_n\) denotes the time needed by the underlying walk to reach
 distance \(n\), then
\[
  \Pbb(D^*\ge n)=\E[G(\tau_n)],
  \qquad
  G(k):=\Pbb(\Xi\ge k).
\]
Since
\[
  G(k)\sim \Gamma(\beta)k^{-\gamma\beta}L(k^\gamma),
\]
and the biased nearest-neighbour random walk has linear hitting-time scale, the off-critical
threshold is \(\beta_c=1/\gamma\). If \(\beta>\beta_c\) and the initial number
of particles per site has finite mean, the model dies out almost surely. If \(\beta<\beta_c\) and the initial configuration is not almost surely empty,
the model survives with positive probability in the direction of the drift.
\end{abstract}

\maketitle

\section{Introduction and preliminaries}

The frog model is an interacting particle system in which particles are placed
on the vertices of a graph. At time zero, some particles are active and the
remaining particles are inactive. Active particles move according to a prescribed
random walk. When an active particle visits a site, all particles located at
that site are activated and start evolving independently according to the same
rules. The main question is whether active particles persist indefinitely or
whether the system eventually dies out.

A foundational study of the frog model is due to \cite{AlvesMachadoPopov2002}. Since then, survival and phase transitions have been studied on several graphs, including lattices and trees; see, for instance, \cite{FontesMachadoSarkar2004,LebensztaynMachadoPopov2005}. On \(\mathbb Z\), the frog model with fixed geometric lifetimes and simple symmetric random walks dies out almost surely. Several related variants have been studied, including models with drift \cite{GantertSchmidt2009}, spatial inhomogeneity \cite{BertacchiMachadoZucca2014}, and random survival parameters \cite{CarvalhoMachado}.

In \cite{CarvalhoMachado}, the authors considered geometric lifetimes with an
i.i.d. random survival parameter \(\pi\in(0,1)\), and obtained a threshold governed
by the mass of \(\pi\) near one. A broader right-edge class of survival-parameter
laws was treated in \cite{CMRGeneral}. The discrete Weibull distribution was introduced by \cite{NakagawaOsaki1975}, and the frog model with discrete Weibull lifetimes, random survival parameter, and simple symmetric random walks was studied in \cite{CMRWeibull}.

Here we consider the same random-survival-parameter mechanism, but with a different underlying motion. Instead of the simple symmetric random walk treated in \cite{CMRWeibull}, each active particle performs a biased nearest-neighbour random walk. Thus the motion is no longer diffusive: the hitting time of a distant site has linear order rather than quadratic order. For the simple symmetric random walk considered in \cite{CMRWeibull}, the corresponding off-critical threshold is \[ \beta_c=\frac{1}{2\gamma}. \] This change of scale is the reason why the critical exponent differs from the symmetric case.

Conditionally on \(\pi=p\), the lifetime has discrete Weibull tail
\[
  \Pbb(\Xi\ge k\mid \pi=p)=p^{k^\gamma},
  \qquad k\in\Nzero,
\]
where \(\gamma>0\). The case \(\gamma=1\) corresponds to geometric lifetimes,
whereas \(\gamma\ne1\) allows the lifetime tail to be either heavier or lighter
than the geometric one. As in \cite{CMRGeneral,CMRWeibull}, the distribution of
the survival parameter is described through its behaviour near the right
endpoint. We assume that its density satisfies
\[
  f_\pi(u)\sim (1-u)^{\beta-1}
  L\left(\frac1{1-u}\right),
  \qquad u\uparrow1,
\]
where \(\beta>0\) and \(L:(0,\infty)\to(0,\infty)\) is a slowly varying function at infinity. The parameter \(\beta\) determines the right-edge order of the density \(f_\pi\).

The proof is based on a one-particle estimate. Let \(S=(S_m)_{m\ge0}\) denote
an underlying walk, let \(D^*\) be the maximal two-sided displacement of a
single particle during its lifetime, and let
\[
  \tau_n:=\inf\{m\ge1: |S_m|=n\}
\]
be the first time at which this walk reaches distance \(n\) from its starting
point. Then
\[
  \Pbb(D^*\ge n)=\E[G(\tau_n)],
  \qquad
  G(k):=\Pbb(\Xi\ge k)=\E[\pi^{k^\gamma}].
\]
The right-edge assumption gives
\[
  G(k)\sim
  \Gamma(\beta)k^{-\gamma\beta}L(k^\gamma),
  \qquad k\to\infty.
\]
Thus the displacement tail is determined by evaluating a regularly varying
lifetime tail at the hitting time of the underlying walk.

For a nearest-neighbour random walk with
\[
  \Pbb(X_1=1)=r,
  \qquad
  \Pbb(X_1=-1)=1-r,
  \qquad r\ne\frac12,
\]
set \(v:=|2r-1|\). Since the walk has nonzero drift, its first exit time from \((-n,n)\) satisfies
\[
  \frac{\tau_n}{n}\longrightarrow \frac1v
  \qquad\text{almost surely}.
\]
Consequently,
\[
  \Pbb(D^*\ge n)
  \sim
  \Gamma(\beta)v^{\gamma\beta}
  n^{-\gamma\beta}L(n^\gamma).
\]
The comparison criterion for the frog model on \(\mathbb Z\) uses \(D^*\) for
extinction and the one-sided displacement in the direction of the drift for
survival. By Proposition~\ref{prop:Dstar-drift}, both give the same
off-critical threshold,
\[
  \beta_c=\frac1\gamma.
\]

 The main conclusion is the following. If \(\beta>1/\gamma\) and
 \(\E(\eta)<\infty\), then the model dies out almost surely. If
 \(\beta<1/\gamma\) and \(\Pbb(\eta=0)<1\), then the model survives with positive
 probability in the direction of the drift. The critical regime
 \(\beta=1/\gamma\) is treated separately in Corollary~\ref{cor:critical-drift-weibull},
 where sufficient extinction and survival conditions are stated in terms of the
 slowly varying factor \(L\) and the constants appearing in the comparison
 criterion.

The paper is organized as follows. Section~\ref{sec:model} defines the model
 and recalls the comparison criterion for the frog model on \(\mathbb Z\).
 Section~\ref{sec:lifetime} proves the lifetime-tail asymptotic generated by
 the random survival parameter. Section~\ref{sec:drift} treats biased
 nearest-neighbour random walks, proves the off-critical threshold
 \(\beta_c=1/\gamma\), and derives the corresponding critical-regime corollary.

\subsection{The model and the comparison criterion}\label{sec:model}

Let \(\eta=(\eta_x)_{x\in\Z}\) be an i.i.d. family of \(\Nzero\)-valued random
variables. The variable \(\eta_x\) represents the number of particles initially
placed at site \(x\). Particles are indexed by pairs \((x,i)\), with
\(1\le i\le \eta_x\).

At time zero, all particles at the origin are active, while all particles at
sites \(x\ne0\) are inactive. 
Fix \(r\in(0,1)\setminus\{1/2\}\). Once activated, each particle follows its
assigned nearest-neighbour random walk with increments satisfying
\[
  \Pbb(X_1=1)=r,
  \qquad
  \Pbb(X_1=-1)=1-r.
\]
The drift is
\[
  \alpha:=\E[X_1]=2r-1,
\]
and we write
\[
  v:=|\alpha|=|2r-1|>0.
\]

Each particle \((x,i)\) also receives an independent survival parameter
\(\pi_{x,i}\in(0,1)\). The variables \((\pi_{x,i})\) are i.i.d. with
common density \(f_\pi\). Conditionally on \(\pi_{x,i}=p\), the lifetime
\(\Xi_{x,i}\in\Nzero\) satisfies
\begin{equation}\label{eq:weibull-life}
  \Pbb(\Xi_{x,i}\ge k\mid \pi_{x,i}=p)
  =
  p^{k^\gamma},
  \qquad k\in\Nzero,
\end{equation}
where \(\gamma>0\). The families \((\eta_x)_{x\in\Z}\),
\((\pi_{x,i})\), and the particle motions are mutually independent.
Conditional on the survival parameters, the lifetimes \((\Xi_{x,i})\) are
independent of one another and independent of the initial occupation numbers
and of the particle motions.

We assume that the survival-parameter density satisfies
\begin{equation}\label{eq:right-edge}
  f_\pi(u)
  \sim
  (1-u)^{\beta-1}
  L\left(\frac1{1-u}\right),
  \qquad u\uparrow1,
\end{equation}
where \(\beta>0\) and \(L:(0,\infty)\to(0,\infty)\) is a slowly varying function at infinity.

To make the particle motions explicit, let
\[
  \bigl(S_m^{x,i}\bigr)_{m\ge0},
  \qquad x\in\Z,\; i\ge1,
\]
be independent copies of the nearest-neighbour random walk specified above,
started from
\[
  S_0^{x,i}=x.
\]
Thus, once particle \((x,i)\) is activated, its position after \(m\) steps is
given by \(S_m^{x,i}\). The family of walks is independent of the initial
occupation numbers, the survival parameters, and the lifetimes.
For a particle initially located at \(x\), define its one-sided and two-sided
maximal displacements during its lifetime by
\[
  D_{x,i}^{\to}
  :=
  \max_{0\le k\le \Xi_{x,i}}
  \bigl(S_k^{x,i}-x\bigr),
  \qquad
  D_{x,i}^{\leftarrow}
  :=
  \max_{0\le k\le \Xi_{x,i}}
  \bigl(x-S_k^{x,i}\bigr),
\]
and
\[
  D_{x,i}^*
  :=
  \max_{0\le k\le \Xi_{x,i}}
  \bigl|S_k^{x,i}-x\bigr|
  =
  D_{x,i}^{\to}\vee D_{x,i}^{\leftarrow}.
\]
By spatial homogeneity, the laws of these random variables do not depend on
\(x\) or \(i\). We write \(D^\to\), \(D^\leftarrow\), and \(D^*\) for
generic copies associated with a particle started from the origin.
\begin{definition}
The frog model survives if, at every time, there is at least one active particle
alive. Otherwise, the model dies out.
\end{definition}

We shall use the following one-dimensional comparison criterion. It is the
standard percolation comparison for the frog model on \(\mathbb Z\), stated
in the form used in \cite[Proposition~1.2]{CarvalhoMachado}. The criterion is
formulated in terms of the one-particle maximal displacements \(D^*\),
\(D^\to\), and \(D^\leftarrow\), and is therefore applicable to the present
lifetime mechanism.

 \begin{proposition}\label{prop:frog-criterion}
 Let \(D^*\) be the maximal two-sided displacement of a single active particle
 and let \(D^\to\) and \(D^\leftarrow\) be the corresponding one-sided maximal
 displacements to the right and to the left.
 \begin{enumerate}[label=\textup{(\roman*)}]
 \item If \(\E(\eta)<\infty\) and
 \[
   \limsup_{n\to\infty} n\Pbb(D^*\ge n)
   <
   \frac{1}{2\E(\eta)},
 \]
 then the frog model dies out almost surely.
 \item If \(\Pbb(\eta=0)<1\) and, for one of the two directions,
 \[
   \liminf_{n\to\infty} n\Pbb(D^\to\ge n)
   >
   \frac{1}{\E(\eta)}
\] or \[
   \liminf_{n\to\infty} n\Pbb(D^\leftarrow\ge n)
   >
   \frac{1}{\E(\eta)},
 \]
 where \(1/\E(\eta):=0\) when \(\E(\eta)=\infty\), then the frog model survives
 with positive probability.
 \end{enumerate}
 \end{proposition}

\subsection{The lifetime tail}\label{sec:lifetime}

We record the lifetime-tail estimate that will be used in the one-particle
argument. The lifetime law and the right-edge assumption have already been
specified in \eqref{eq:weibull-life} and \eqref{eq:right-edge}. For a generic
particle, define
\[
  G(k):=\Pbb(\Xi\ge k)
  =
  \E[\pi^{k^\gamma}],
  \qquad k\in\Nzero.
\]
Equivalently,
\[
  G(k)=\int_0^1 u^{k^\gamma}f_\pi(u)\dd u,
  \qquad k\in\Nzero.
\]
We also use the convention
\[
  G(\infty):=0.
\]

\begin{lemma}\label{lem:lifetime-tail}
Assume \eqref{eq:right-edge}. Then
\[
  G(k)
  \sim
  \Gamma(\beta)k^{-\gamma\beta}L(k^\gamma),
  \qquad k\to\infty.
\]
In particular, \(G\) is nonincreasing and regularly varying at infinity with
index \(-\gamma\beta\).
\end{lemma}

\begin{proof}
We first prove the following estimate:
\begin{equation}\label{eq:abelian-step}
  \int_0^1 u^a f_\pi(u)\dd u
  \sim
  \Gamma(\beta)a^{-\beta}L(a),
  \qquad a\to\infty.
\end{equation}
Afterwards we shall take \(a=k^\gamma\).

Let \(A\in(0,1)\). We split
\[
  \int_0^1 u^a f_\pi(u)\dd u
  =
  \int_0^A u^a f_\pi(u)\dd u
  +
  \int_A^1 u^a f_\pi(u)\dd u.
\]
The first term is negligible. Indeed,
\[
  0\le
  \int_0^A u^a f_\pi(u)\dd u
  \le
  A^a\int_0^A f_\pi(u)\dd u
  \le A^a.
\]
Since \(A\in(0,1)\), \(A^a\) decays exponentially, whereas
\(a^{-\beta}L(a)\) is regularly varying. Hence
\[
  \int_0^A u^a f_\pi(u)\dd u
  =
  o\left(a^{-\beta}L(a)\right).
\]

It remains to analyze the contribution from \(u\) close to \(1\). Make the
change of variables
\[
  y=a(1-u),
  \qquad
  u=1-\frac ya,
  \qquad
  \dd u=-\frac1a\dd y.
\]
Then
\[
  \int_A^1 u^a f_\pi(u)\dd u
  =
  \frac1a
  \int_0^{a(1-A)}
  \left(1-\frac ya\right)^a
  f_\pi\left(1-\frac ya\right)\dd y.
\]
Multiplying by \(a^\beta/L(a)\), we obtain
\begin{equation}\label{eq:scaled-integral}
  \frac{a^\beta}{L(a)}
  \int_A^1 u^a f_\pi(u)\dd u
  =
  \int_0^{a(1-A)}
  \left(1-\frac ya\right)^a
  y^{\beta-1}
  R_a(y)\dd y,
\end{equation}
where
\[
  R_a(y)
  :=
  \frac{
  f_\pi(1-y/a)
  }{
  (y/a)^{\beta-1}L(a)
  }
  =
  \frac{
  f_\pi(1-y/a)
  }{
  (y/a)^{\beta-1}L(a/y)
  }
  \frac{L(a/y)}{L(a)}.
\]
For each fixed \(y>0\), as \(a\to\infty\), we have \(1-y/a\uparrow1\). Hence,
by \eqref{eq:right-edge},
\[
  \frac{
  f_\pi(1-y/a)
  }{
  (y/a)^{\beta-1}L(a/y)
  }
  \longrightarrow 1.
\]
Also, since \(L\) is slowly varying,
\[
  \frac{L(a/y)}{L(a)}
  \longrightarrow 1.
\]
Therefore,
\[
  R_a(y)\longrightarrow 1,
  \qquad y>0.
\]
Moreover,
\[
  \left(1-\frac ya\right)^a\longrightarrow e^{-y},
  \qquad y>0.
\]
Thus the integrand in \eqref{eq:scaled-integral} converges pointwise to
\[
  e^{-y}y^{\beta-1}.
\]

We now justify dominated convergence. Choose \(0<\varepsilon<\beta\).
By Potter’s bound for slowly varying functions
\cite[Theorem~1.5.6]{BinghamGoldieTeugels1987}, for every \(C_P>1\)
there exists \(x_0=x_0(\varepsilon,C_P)\ge1\) such that
\[
  \frac{L(z)}{L(x)}
  \le
  C_P
  \max\left\{
    \left(\frac{z}{x}\right)^\varepsilon,
    \left(\frac{z}{x}\right)^{-\varepsilon}
  \right\},
  \qquad x,z\ge x_0.
\]
Thus, with \(x=a\) and \(z=a/y\), for all \(a\ge x_0\) and
\(0<y\le a/x_0\),
\begin{equation}\label{ecu_L}
  \frac{L(a/y)}{L(a)}
  \le
  C_P\left(y^\varepsilon+y^{-\varepsilon}\right).
\end{equation}
By \eqref{eq:right-edge}, there exist \(A_f\in(0,1)\) and \(C_f>0\)
such that
\[
  f_\pi(u)
  \le
  C_f(1-u)^{\beta-1}
  L\left(\frac1{1-u}\right),
  \qquad A_f<u<1.
\]
We now take \(A\in(0,1)\) sufficiently close to \(1\) so that
\[
  A>A_f,
  \qquad
  1-A\le x_0^{-1}.
\]
Then, for all \(a\ge x_0\) and \(0<y\le a(1-A)\), both
\(a/y\ge x_0\) and \(1-y/a\ge A>A_f\). Hence
\[
  f_\pi\left(1-\frac ya\right)
  \le
  C_f\left(\frac ya\right)^{\beta-1}L\left(\frac ay\right),
\]
and \eqref{ecu_L} holds on \((0,a(1-A)]\). Therefore
\[
  R_a(y)
  \le
  C_fC_P\left(y^\varepsilon+y^{-\varepsilon}\right),
  \qquad 0<y\le a(1-A).
\]
Since \((1-y/a)^a\le e^{-y}\) for \(0<y<a\), the integrand in
\eqref{eq:scaled-integral}, extended by zero outside \((0,a(1-A)]\), is
bounded by
\[
C e^{-y}
  \left(y^{\beta-1+\varepsilon}+y^{\beta-1-\varepsilon}\right),
\]
which is integrable on \((0,\infty)\). The dominated convergence theorem gives
\[
  \frac{a^\beta}{L(a)}
  \int_A^1 u^a f_\pi(u)\,du
  \longrightarrow
  \int_0^\infty e^{-y}y^{\beta-1}\,dy
  =
  \Gamma(\beta).
\]
Together with the negligibility of the integral over \((0,A)\), this proves
\eqref{eq:abelian-step}.

Now take
\[
  a=k^\gamma.
\]
Since \(\gamma>0\), we have \(a\to\infty\) as \(k\to\infty\). Therefore, by
\eqref{eq:abelian-step},
\[
  G(k)
  =
  \int_0^1 u^{k^\gamma}f_\pi(u)\dd u
  \sim
  \Gamma(\beta)(k^\gamma)^{-\beta}L(k^\gamma).
\]
Equivalently,
\[
  G(k)
  \sim
  \Gamma(\beta)k^{-\gamma\beta}L(k^\gamma).
\]

It remains to record monotonicity and regular variation. Since \(0<u<1\) and
\(k\mapsto k^\gamma\) is increasing,
\[
  u^{(k+1)^\gamma}\le u^{k^\gamma},
  \qquad k\in\Nzero.
\]
Integrating with respect to \(f_\pi(u)\dd u\), we obtain
\[
  G(k+1)\le G(k),
\]
so \(G\) is nonincreasing.

Finally, the function
\[
  k\mapsto k^{-\gamma\beta}L(k^\gamma)
\]
is regularly varying with index \(-\gamma\beta\), because \(L(k^\gamma)\) is
slowly varying. Since \(G(k)\) is asymptotic to a positive constant times this
function, \(G\) is regularly varying with index \(-\gamma\beta\).
\end{proof}

\begin{lemma}\label{lem:basic-identity}
Let \(S=(S_m)_{m\ge0}\) be a nearest-neighbour motion on \(\mathbb Z\), started
from \(S_0=0\), and let \(\Xi\) be an independent lifetime with tail \(G\).
Define
\[
  D^*:=\max_{0\le k\le \Xi}|S_k|.
\]
Let
\[
  \tau_n:=\inf\{m\ge1: |S_m|=n\},
  \qquad n\ge1,
\]
with the convention \(\inf\varnothing=\infty\). Then
\[
  \Pbb(D^*\ge n)=\E[G(\tau_n)].
\]
If
\[
  D^\to:=\max_{0\le k\le\Xi}S_k,
  \qquad
  \tau_n^+:=\inf\{m\ge1:S_m=n\},
\]
then
\[
  \Pbb(D^\to\ge n)=\E[G(\tau_n^+)].
\]
Analogously, if
\[
  D^\leftarrow:=\max_{0\le k\le\Xi}(-S_k),
  \qquad
  \tau_n^-:=\inf\{m\ge1:S_m=-n\},
\]
then
\[
  \Pbb(D^\leftarrow\ge n)=\E[G(\tau_n^-)].
\]
In the last two identities, we again use the convention \(G(\infty)=0\).
\end{lemma}

\begin{proof}
We prove the identity for \(D^*\). Since the motion is nearest-neighbour and
starts from \(0\), the event that the path reaches distance at least \(n\)
before time \(\Xi\) is exactly the event that the first hitting time of
\(\{-n,n\}\) is not larger than the lifetime. Thus
\[
\{D^*\ge n\}=\{\tau_n\le \Xi\}.
\]
Indeed, if \(D^*\ge n\), then there exists \(0\le k\le \Xi\) such that
\(|S_k|\ge n\). Since \(S\) is nearest-neighbour and \(S_0=0\), the path
must hit either \(n\) or \(-n\) at some time not larger than \(k\). Hence
\(\tau_n\le k\le \Xi\). Conversely, if \(\tau_n\le \Xi\), then
\(|S_{\tau_n}|=n\), and therefore \(D^*\ge n\).
Using this equivalence and the independence between \(\Xi\) and \(S\),
\[
  \Pbb(D^*\ge n)
  =
  \Pbb(\tau_n\le \Xi)
  =
  \E\left[\Pbb(\Xi\ge \tau_n\mid \tau_n)\right].
\]
If \(\tau_n<\infty\), then
\[
  \Pbb(\Xi\ge \tau_n\mid \tau_n)=G(\tau_n).
\]
If \(\tau_n=\infty\), then \(\{\Xi\ge \tau_n\}\) is empty because
\(\Xi\) is finite-valued, and by convention \(G(\infty)=0\). Therefore,
\[
  \Pbb(D^*\ge n)=\E[G(\tau_n)].
\]
The one-sided identities follow in the same way from
\(\{D^\to\ge n\}=\{\tau_n^+\le \Xi\}\) and
\(\{D^\leftarrow\ge n\}=\{\tau_n^-\le \Xi\}\).
\end{proof}

\section{Biased nearest-neighbour random walks}\label{sec:drift} 

In this section we record hitting-time estimates for a generic copy of the
biased nearest-neighbour random walk specified in Section~\ref{sec:model}. Let
\[
  S_n:=X_1+\cdots+X_n,
  \qquad S_0:=0,
\]
where \[
  \Pbb(X_1=1)=r,
  \qquad
  \Pbb(X_1=-1)=1-r,
\]
with \(r\in(0,1)\setminus\{1/2\}\). Recall that
\[
  \alpha:=\E[X_1]=2r-1,
  \qquad
  v:=|\alpha|=|2r-1|>0.
\]

Thus \(\alpha>0\) when \(r>1/2\), and \(\alpha<0\) when \(r<1/2\).

For \(n\ge1\), define the two-sided hitting time
\[
  \tau_n^*:=\inf\{m\ge1: |S_m|=n\}.
\]
When \(r>1/2\), we shall also use the hitting time of \(n\),
\[
  \tau_n^+:=\inf\{m\ge1:S_m=n\}.
\]
When \(r<1/2\), we shall use the hitting time of \(-n\),
\[
  \tau_n^-:=\inf\{m\ge1:S_m=-n\}.
\]

The first result identifies the order of these hitting times. Since the walk
has nonzero drift, the relevant scale is linear in \(n\). This is a standard
consequence of the strong law of large numbers.

\begin{lemma}\label{lem:drifted-hitting-time}
The following hitting-time asymptotics hold:
\[
  \frac{\tau_n^*}{n}
  \xrightarrow[n\to\infty]{a.s.}
  \frac1v.
\]
Moreover, if \(r>1/2\), then
\[
  \frac{\tau_n^+}{n}
  \xrightarrow[n\to\infty]{a.s.}
  \frac1{2r-1},
\]
whereas, if \(r<1/2\), then
\[
  \frac{\tau_n^-}{n}
  \xrightarrow[n\to\infty]{a.s.}
  \frac1{1-2r}.
\]
\end{lemma}

\begin{proof}
It is enough to give the argument for \(r>1/2\), and then use reflection for
the case \(r<1/2\). Assume first that \(r>1/2\), and set \(  \alpha=2r-1>0\). By the strong law of large numbers,
\begin{equation}\label{eq:slln-drift-positive}
  \frac{S_m}{m}\longrightarrow \alpha
  \qquad\text{a.s.}
\end{equation}
We work on an event of probability one on which \eqref{eq:slln-drift-positive}
holds.

We first prove the asymptotic for \(\tau_n^+\). Let \(b>1/\alpha\). Choose
\(\varepsilon>0\) such that \(b(\alpha-\varepsilon)>1\). By
\eqref{eq:slln-drift-positive}, there exists \(m_0=m_0(\varepsilon)\) such that
\[
  S_m\ge (\alpha-\varepsilon)m,
  \qquad m\ge m_0 .
\]
Put \(m_n:=\lfloor bn\rfloor\). For all sufficiently large \(n\),
\(m_n\ge m_0\), and
\[
  \frac{(\alpha-\varepsilon)m_n}{n}
  \longrightarrow
  b(\alpha-\varepsilon)>1.
\]
Hence \(S_{m_n}\ge(\alpha-\varepsilon)m_n>n\) for all sufficiently large \(n\).
Since the walk is nearest-neighbour and starts from \(0\), it must have visited
\(n\) by time \(m_n\). Therefore \(\tau_n^+\le m_n\) eventually, and
\[
  \limsup_{n\to\infty}\frac{\tau_n^+}{n}\le b.
\]
Letting \(b\downarrow1/\alpha\), we obtain
\begin{equation}\label{eq:tau-plus-limsup}
  \limsup_{n\to\infty}\frac{\tau_n^+}{n}
  \le
  \frac1\alpha .
\end{equation}

For the lower bound, let \(a<1/\alpha\). If \(a\le1\), then
\(\tau_n^+\ge n>an\) for every \(n\). Assume therefore that \(a>1\), and choose
\(\varepsilon>0\) such that
\[
  \alpha+\varepsilon<\frac1a.
\]
Again by \eqref{eq:slln-drift-positive}, there exists
\(m_1=m_1(\varepsilon)\) such that
\[
  \frac{S_m}{m}\le \alpha+\varepsilon,
  \qquad m\ge m_1 .
\]
Suppose that, along an infinite subsequence \((n_j^{(1)})_{j\ge1}\),
\[
  \tau_{n_j^{(1)}}^+\le a n_j^{(1)} .
\]
Since the walk is nearest-neighbour, \(\tau_{n_j^{(1)}}^+\ge n_j^{(1)}\), and
therefore \(\tau_{n_j^{(1)}}^+\to\infty\). Thus, after discarding finitely many
indices, \(\tau_{n_j^{(1)}}^+\ge m_1\). For such \(j\),
\[
  S_{\tau_{n_j^{(1)}}^+}=n_j^{(1)},
\]
and hence
\[
  \frac{S_{\tau_{n_j^{(1)}}^+}}{\tau_{n_j^{(1)}}^+}
  =
  \frac{n_j^{(1)}}{\tau_{n_j^{(1)}}^+}
  \ge
  \frac1a
  >
  \alpha+\varepsilon,
\]
which contradicts the definition of \(m_1\). Hence \(\tau_n^+>an\) eventually,
and
\[
  \liminf_{n\to\infty}\frac{\tau_n^+}{n}\ge a.
\]
Letting \(a\uparrow1/\alpha\), we get
\begin{equation}\label{eq:tau-plus-liminf}
  \liminf_{n\to\infty}\frac{\tau_n^+}{n}
  \ge
  \frac1\alpha .
\end{equation}
Combining \eqref{eq:tau-plus-limsup} and \eqref{eq:tau-plus-liminf} gives
\[
  \frac{\tau_n^+}{n}
  \longrightarrow
  \frac1\alpha
  =
  \frac1{2r-1}.
\]

We now prove the corresponding assertion for \(\tau_n^*\). Since
\(\tau_n^*\le\tau_n^+\), the upper bound gives
\[
  \limsup_{n\to\infty}\frac{\tau_n^*}{n}
  \le
  \frac1\alpha.
\]
For the lower bound, let \(a<1/\alpha\). If \(a\le1\), then
\(\tau_n^*\ge n>an\). Assume \(a>1\), and choose \(\varepsilon>0\) such that
\[
  \alpha+\varepsilon<\frac1a,
  \qquad
  \alpha-\varepsilon>0.
\]
By \eqref{eq:slln-drift-positive}, there exists \(m_2=m_2(\varepsilon)\) such
that
\[
  \alpha-\varepsilon
  \le
  \frac{S_m}{m}
  \le
  \alpha+\varepsilon,
  \qquad m\ge m_2 .
\]
Suppose that, along an infinite subsequence \((n_j^{(2)})_{j\ge1}\),
\[
  \tau_{n_j^{(2)}}^*\le a n_j^{(2)} .
\]
Since \(\tau_{n_j^{(2)}}^*\ge n_j^{(2)}\), we have
\(\tau_{n_j^{(2)}}^*\to\infty\). After discarding finitely many terms, assume
\[
  \tau_{n_j^{(2)}}^*\ge m_2,
  \qquad j\ge1.
\]
For each \(j\),
\[
  S_{\tau_{n_j^{(2)}}^*}=n_j^{(2)}
  \qquad\text{or}\qquad
  S_{\tau_{n_j^{(2)}}^*}=-n_j^{(2)}.
\]
Define
\[
  J_+:=\left\{
    j\ge1:
    S_{\tau_{n_j^{(2)}}^*}=n_j^{(2)}
  \right\},
  \qquad
  J_-:=\left\{
    j\ge1:
    S_{\tau_{n_j^{(2)}}^*}=-n_j^{(2)}
  \right\}.
\]
Since \(\mathbb N=J_+\cup J_-\), at least one of \(J_+\) and \(J_-\) is
infinite.

If \(J_+\) is infinite, take a further infinite subsequence
\((n_\ell^{(2,+)})_{\ell\ge1}\). Then
\[
  \frac{S_{\tau_{n_\ell^{(2,+)}}^*}}{\tau_{n_\ell^{(2,+)}}^*}
  =
  \frac{n_\ell^{(2,+)}}{\tau_{n_\ell^{(2,+)}}^*}
  \ge
  \frac1a
  >
  \alpha+\varepsilon,
\]
contradicting the upper bound \(S_m/m\le\alpha+\varepsilon\) for \(m\ge m_2\).
If \(J_-\) is infinite, take a further infinite subsequence
\((n_\ell^{(2,-)})_{\ell\ge1}\). Then
\[
  \frac{S_{\tau_{n_\ell^{(2,-)}}^*}}{\tau_{n_\ell^{(2,-)}}^*}
  =
  -\frac{n_\ell^{(2,-)}}{\tau_{n_\ell^{(2,-)}}^*}
  <0,
\]
contradicting the lower bound \(S_m/m\ge\alpha-\varepsilon>0\) for
\(m\ge m_2\). Both cases are impossible. Therefore \(\tau_n^*>an\)
eventually, and
\[
  \liminf_{n\to\infty}\frac{\tau_n^*}{n}\ge a.
\]
Letting \(a\uparrow1/\alpha\), we obtain
\[
  \liminf_{n\to\infty}\frac{\tau_n^*}{n}
  \ge
  \frac1\alpha.
\]
Together with the upper bound, this proves
\[
  \frac{\tau_n^*}{n}
  \longrightarrow
  \frac1\alpha
  =
  \frac1{2r-1}
  =
  \frac1v .
\]

It remains to treat \(r<1/2\). Define \(\widetilde S_m:=-S_m\). Then
\((\widetilde S_m)_{m\ge0}\) is a nearest-neighbour random walk with
\[
  \mathbb P(\widetilde X_1=1)=1-r>\frac12,
  \qquad
  \mathbb P(\widetilde X_1=-1)=r,
\]
and drift \(1-2r=v>0\). Applying the previous part to
\((\widetilde S_m)_{m\ge0}\) yields
\[
  \frac1n\inf\{m\ge1:\widetilde S_m=n\}
  \longrightarrow
  \frac1{1-2r}.
\]
Since
\[
  \inf\{m\ge1:\widetilde S_m=n\}
  =
  \inf\{m\ge1:S_m=-n\}
  =
  \tau_n^-,
\]
we get
\[
  \frac{\tau_n^-}{n}
  \longrightarrow
  \frac1{1-2r}.
\]
Moreover,
\[
  \inf\{m\ge1:|\widetilde S_m|=n\}
  =
  \inf\{m\ge1:|S_m|=n\}
  =
  \tau_n^*,
\]
and therefore
\[
  \frac{\tau_n^*}{n}
  \longrightarrow
  \frac1{1-2r}
  =
  \frac1v .
\]
\end{proof}

We next show how the linear hitting-time asymptotic transfers to the tail of
the lifetime evaluated at the hitting time.

\begin{lemma}\label{lem:transfer-drift}
Let \(G\) be positive, nonincreasing, and regularly varying at infinity with
index \(-\theta\), where \(\theta>0\). Then
\[
  \frac{\E[G(\tau_n^*)]}{G(n)}
  \longrightarrow
  v^\theta.
\]
If \(r>1/2\), then also
\[
  \frac{\E[G(\tau_n^+)]}{G(n)}
  \longrightarrow
  (2r-1)^\theta.
\]
If \(r<1/2\), then
\[
  \frac{\E[G(\tau_n^-)]}{G(n)}
  \longrightarrow
  (1-2r)^\theta.
\]
\end{lemma}

\begin{proof}
We prove the result for \(\tau_n^*\). The one-sided statements are proved by
the same argument, using the corresponding part of
Lemma~\ref{lem:drifted-hitting-time}.

By Lemma~\ref{lem:drifted-hitting-time},
\[
  \frac{\tau_n^*}{n}
  \longrightarrow
  \frac1v
  \qquad\text{a.s.}
\]

Let
\[
  c:=\frac1v.
\]
Fix \(0<\delta<c\). Since
\[
  \frac{\tau_n^*}{n}\longrightarrow c
  \qquad\text{a.s.},
\]
for all sufficiently large \(n\),
\[
  n(c-\delta)\le \tau_n^*\le n(c+\delta).
\]
Since \(G\) is nonincreasing,
\[
  G(\lceil n(c+\delta)\rceil)
  \le
  G(\tau_n^*)
  \le
  G(\lfloor n(c-\delta)\rfloor)
\]
for all sufficiently large \(n\). Dividing by \(G(n)\), we get
\[
  \frac{G(\lceil n(c+\delta)\rceil)}{G(n)}
  \le
  \frac{G(\tau_n^*)}{G(n)}
  \le
  \frac{G(\lfloor n(c-\delta)\rfloor)}{G(n)}.
\]
By regular variation of \(G\) with index \(-\theta\),
\[
  \frac{G(\lceil n(c+\delta)\rceil)}{G(n)}
  \longrightarrow
  (c+\delta)^{-\theta}
\]
and
\[
  \frac{G(\lfloor n(c-\delta)\rfloor)}{G(n)}
  \longrightarrow
  (c-\delta)^{-\theta}.
\]
Indeed, for \(k_n=\lceil n(c+\delta)\rceil\) or
\(k_n=\lfloor n(c-\delta)\rfloor\), one has \(k_n/n\to c+\delta\) or
\(k_n/n\to c-\delta\), respectively. Hence regular variation gives
\[
  \frac{G(k_n)}{G(n)}
  =
  \left(\frac{k_n}{n}\right)^{-\theta}
  \frac{\ell(k_n)}{\ell(n)}
  \longrightarrow
  x^{-\theta},
\]
where \(x=c+\delta\) or \(x=c-\delta\), and \(G(n)=n^{-\theta}\ell(n)\) with
\(\ell\) slowly varying.

Consequently,
\[
  (c+\delta)^{-\theta}
  \le
  \liminf_{n\to\infty}\frac{G(\tau_n^*)}{G(n)}
  \le
  \limsup_{n\to\infty}\frac{G(\tau_n^*)}{G(n)}
  \le
  (c-\delta)^{-\theta}.
\]
Letting \(\delta\downarrow0\), we obtain

\[
  \frac{G(\tau_n^*)}{G(n)}
  \longrightarrow
  c^{-\theta}
  =
  v^\theta
  \qquad\text{a.s.}
\]

It remains to pass to expectations. Since the walk is nearest-neighbour and
starts from \(0\), it cannot reach distance \(n\) before time \(n\). Hence
\[
  \tau_n^*\ge n.
\]
Since \(G\) is nonincreasing,
\[
  0\le
  G(\tau_n^*)\le G(n).
\]
Therefore,
\[
  0\le
  \frac{G(\tau_n^*)}{G(n)}
  \le 1.
\]
The dominated convergence theorem gives
\[
  \frac{\E[G(\tau_n^*)]}{G(n)}
  =
  \E\left[\frac{G(\tau_n^*)}{G(n)}\right]
  \longrightarrow
  v^\theta.
\]

If \(r>1/2\), Lemma~\ref{lem:drifted-hitting-time} gives
\[
  \frac{\tau_n^+}{n}
  \longrightarrow
  \frac1{2r-1}
  \qquad\text{a.s.}
\]
Also \(\tau_n^+\ge n\). Repeating the previous argument gives
\[
  \frac{\E[G(\tau_n^+)]}{G(n)}
  \longrightarrow
  (2r-1)^\theta.
\]
If \(r<1/2\), the same argument applied to \(\tau_n^-\) gives
\[
  \frac{\E[G(\tau_n^-)]}{G(n)}
  \longrightarrow
  (1-2r)^\theta.
\]
The proof is complete.
\end{proof}

We now apply the previous transfer lemma to the one-particle displacement
tail.

\begin{proposition}\label{prop:Dstar-drift}
Assume condition \eqref{eq:right-edge}. For the biased nearest-neighbour random walk,
\[
  \Pbb(D^*\ge n)
  \sim
  \Gamma(\beta)v^{\gamma\beta}
  n^{-\gamma\beta}L(n^\gamma),
  \qquad n\to\infty.
\]
If \(r>1/2\), then
\[
  \Pbb(D^\to\ge n)
  \sim
  \Gamma(\beta)(2r-1)^{\gamma\beta}
  n^{-\gamma\beta}L(n^\gamma).
\]
If \(r<1/2\), then
\[
  \Pbb(D^\leftarrow\ge n)
  \sim
  \Gamma(\beta)(1-2r)^{\gamma\beta}
  n^{-\gamma\beta}L(n^\gamma).
\]
\end{proposition}

\begin{proof}
By Lemma~\ref{lem:basic-identity},
\[
  \Pbb(D^*\ge n)=\E[G(\tau_n^*)].
\]
By Lemma~\ref{lem:lifetime-tail},
\[
  G(k)
  \sim
  \Gamma(\beta)k^{-\gamma\beta}L(k^\gamma),
  \qquad k\to\infty.
\]
In particular, \(G\) is regularly varying with index
\[
  -\theta=-\gamma\beta,
\]
that is,
\[
  \theta=\gamma\beta.
\]
Applying Lemma~\ref{lem:transfer-drift}, we obtain
\[
  \E[G(\tau_n^*)]
  \sim
  v^{\gamma\beta}G(n).
\]
Using the asymptotic form of \(G(n)\), again from
Lemma~\ref{lem:lifetime-tail},
\[
  G(n)
  \sim
  \Gamma(\beta)n^{-\gamma\beta}L(n^\gamma).
\]
Therefore,
\[
  \Pbb(D^*\ge n)
  =
  \E[G(\tau_n^*)]
  \sim
  \Gamma(\beta)v^{\gamma\beta}
  n^{-\gamma\beta}L(n^\gamma).
\]

If \(r>1/2\), then Lemma~\ref{lem:basic-identity} gives
\[
  \Pbb(D^\to\ge n)=\E[G(\tau_n^+)].
\]
Applying the one-sided part of Lemma~\ref{lem:transfer-drift} with
\(\theta=\gamma\beta\),
\[
  \E[G(\tau_n^+)]
  \sim
  (2r-1)^{\gamma\beta}G(n).
\]
Consequently,
\[
  \Pbb(D^\to\ge n)
  \sim
  \Gamma(\beta)(2r-1)^{\gamma\beta}
  n^{-\gamma\beta}L(n^\gamma).
\]

If \(r<1/2\), the same argument with \(\tau_n^-\) and \(D^\leftarrow\) gives
\[
  \Pbb(D^\leftarrow\ge n)
  \sim
  \Gamma(\beta)(1-2r)^{\gamma\beta}
  n^{-\gamma\beta}L(n^\gamma).
\]
The proof is complete.
\end{proof}

We now state the corresponding off-critical extinction--survival result for the frog model with biased nearest-neighbour random walks.

\begin{theorem}\label{thm:drift-weibull-frog}
Assume \eqref{eq:right-edge}. Consider the frog model on \(\mathbb Z\) in which activated particles perform
independent biased nearest-neighbour random walks with
\[
  \Pbb(X_1=1)=r,
  \qquad
  \Pbb(X_1=-1)=1-r,
\]
where \(r\in(0,1)\setminus\{1/2\}\). Set
\[
  \beta_c:=\frac1\gamma.
\]
Then:
\begin{enumerate}[label=\textup{(\roman*)}]
\item If \(\E(\eta)<\infty\) and \(\beta>\beta_c\), then the frog model dies
out almost surely.
\item If \(\Pbb(\eta=0)<1\) and \(\beta<\beta_c\), then the frog model survives
with positive probability in the direction of the drift, that is, to the right
if \(r>1/2\) and to the left if \(r<1/2\).
\end{enumerate}
\end{theorem}

\begin{proof}
By Proposition~\ref{prop:Dstar-drift},
\[
  \Pbb(D^*\ge n)
  \sim
  \Gamma(\beta)v^{\gamma\beta}
  n^{-\gamma\beta}L(n^\gamma).
\]
Hence
\[
  n\Pbb(D^*\ge n)
  \sim
  \Gamma(\beta)v^{\gamma\beta}
  n^{1-\gamma\beta}L(n^\gamma).
\]

We first consider the case
\[
  \beta>\frac1\gamma .
\]
Then \(1-\gamma\beta<0\). Since \(L\) is slowly varying and positive,
\[
  x^{-\rho}L(x)\longrightarrow 0
  \qquad\text{for every }\rho>0.
\]
Taking \(x=n^\gamma\) and
\[
  \rho:=\frac{\gamma\beta-1}{\gamma}>0,
\]
we obtain
\[
  n^{1-\gamma\beta}L(n^\gamma)
  =
  (n^\gamma)^{-\rho}L(n^\gamma)
  \longrightarrow 0.
\]
Therefore,
\[
  n\Pbb(D^*\ge n)\longrightarrow0.
\]
By Proposition~\ref{prop:frog-criterion}(i), the frog model dies out almost
surely whenever \(\E(\eta)<\infty\).

We now consider the case
\[
  \beta<\frac1\gamma .
\]
Then \(1-\gamma\beta>0\). Again, since \(L\) is slowly varying and positive,
\[
  x^\rho L(x)\longrightarrow\infty
  \qquad\text{for every }\rho>0.
\]
Taking \(x=n^\gamma\) and
\[
  \rho:=\frac{1-\gamma\beta}{\gamma}>0,
\]
we get
\[
  n^{1-\gamma\beta}L(n^\gamma)
  =
  (n^\gamma)^\rho L(n^\gamma)
  \longrightarrow\infty.
\]

If \(r>1/2\), Proposition~\ref{prop:Dstar-drift} gives
\[
  \Pbb(D^\to\ge n)
  \sim
  \Gamma(\beta)(2r-1)^{\gamma\beta}
  n^{-\gamma\beta}L(n^\gamma).
\]
Consequently,
\[
  n\Pbb(D^\to\ge n)
  \sim
  \Gamma(\beta)(2r-1)^{\gamma\beta}
  n^{1-\gamma\beta}L(n^\gamma)
  \longrightarrow\infty.
\]
By Proposition~\ref{prop:frog-criterion}(ii), the frog model survives with
positive probability in the direction of the drift.

If \(r<1/2\), then Proposition~\ref{prop:Dstar-drift} gives
\[
  \Pbb(D^\leftarrow\ge n)
  \sim
  \Gamma(\beta)(1-2r)^{\gamma\beta}
  n^{-\gamma\beta}L(n^\gamma).
\]
Thus
\[
  n\Pbb(D^\leftarrow\ge n)
  \sim
  \Gamma(\beta)(1-2r)^{\gamma\beta}
  n^{1-\gamma\beta}L(n^\gamma)
  \longrightarrow\infty.
\]
Again, Proposition~\ref{prop:frog-criterion}(ii) gives survival with positive
probability in the direction of the drift.
\end{proof}

\begin{remark}
When \(\gamma=1\), the conditional lifetime is geometric and Theorem~\ref{thm:drift-weibull-frog}
gives the critical value \(\beta_c=1\) for biased nearest-neighbour walks.
\end{remark}

\begin{corollary}\label{cor:critical-drift-weibull}
 Assume \eqref{eq:right-edge} and let \(\beta=1/\gamma\). Set
 \(v:=|2r-1|\).
 \begin{enumerate}[label=\textup{(\roman*)}]
 \item If \(\E(\eta)<\infty\) and
 \[
   \Gamma(1/\gamma)v\,
   \limsup_{n\to\infty} L(n^\gamma)
   <
   \frac{1}{2\E(\eta)},
 \]
 then the frog model dies out almost surely.
 \item If \(\Pbb(\eta=0)<1\) and
 \[
   \Gamma(1/\gamma)v\,
   \liminf_{n\to\infty} L(n^\gamma)
   >
   \frac{1}{\E(\eta)},
 \]
 where \(1/\E(\eta):=0\) when \(\E(\eta)=\infty\), then the frog model survives
 with positive probability in the direction of the drift. More precisely, the
 survival occurs to the right if \(r>1/2\), and to the left if \(r<1/2\).
 \end{enumerate}
 \end{corollary}
 \begin{proof}
 Since \(\beta=1/\gamma\), Proposition~\ref{prop:Dstar-drift} gives
 \[
   n\Pbb(D^*\ge n)
   \sim
   \Gamma(1/\gamma)v\,L(n^\gamma).
 \]
 Therefore, if
 \[
   \Gamma(1/\gamma)v\,
   \limsup_{n\to\infty}L(n^\gamma)
   <
   \frac{1}{2\E(\eta)},
\] then \[
   \limsup_{n\to\infty} n\Pbb(D^*\ge n)
   <
   \frac{1}{2\E(\eta)}.
 \]
 Proposition~\ref{prop:frog-criterion}\textup{(i)} implies almost sure
 extinction.
 We now prove the survival statement. Suppose first that \(r>1/2\). Again by
 Proposition~\ref{prop:Dstar-drift},
 \[
   n\Pbb(D^\to\ge n)
   \sim
   \Gamma(1/\gamma)(2r-1)L(n^\gamma)
   =
   \Gamma(1/\gamma)vL(n^\gamma).
 \]
 Hence the assumed lower bound on the \(\liminf\) gives
 \[
   \liminf_{n\to\infty}n\Pbb(D^\to\ge n)
   >
   \frac{1}{\E(\eta)}.
 \]
 Proposition~\ref{prop:frog-criterion}\textup{(ii)} yields survival with
 positive probability to the right.
 If \(r<1/2\), then Proposition~\ref{prop:Dstar-drift} gives
 \[
   n\Pbb(D^\leftarrow\ge n)
   \sim
   \Gamma(1/\gamma)(1-2r)L(n^\gamma)
   =
   \Gamma(1/\gamma)vL(n^\gamma).
 \]

The same lower-bound assumption therefore implies
\[
  \liminf_{n\to\infty}n\Pbb(D^\leftarrow\ge n)
  >
  \frac{1}{\E(\eta)}.
\]
Proposition~\ref{prop:frog-criterion}\textup{(ii)} yields survival with
positive probability to the left.
\end{proof}

\begin{remark}
At the critical value \(\beta=1/\gamma\), Proposition~\ref{prop:Dstar-drift}
gives
\[
  n\Pbb(D^*\ge n)\sim \Gamma(1/\gamma)vL(n^\gamma).
\]
The same asymptotic holds for \(n\Pbb(D^\to\ge n)\) if \(r>1/2\), and for
\(n\Pbb(D^\leftarrow\ge n)\) if \(r<1/2\). Therefore,
Proposition~\ref{prop:frog-criterion} gives extinction under the upper
condition in Corollary~\ref{cor:critical-drift-weibull} and survival under the
lower condition in Corollary~\ref{cor:critical-drift-weibull}. When
\(\E(\eta)<\infty\), the remaining critical region is
\[
  \left\{
  \frac{1}{2\E(\eta)}
  \le
  \Gamma(1/\gamma)v\limsup_{n\to\infty}L(n^\gamma)
  \right\}
  \cap
  \left\{
  \Gamma(1/\gamma)v\liminf_{n\to\infty}L(n^\gamma)
  \le
  \frac{1}{\E(\eta)}
  \right\}.
\]
This region is not decided by Proposition~\ref{prop:frog-criterion}.
\end{remark}

Figure~\ref{fig:diagrama_Weibull} summarizes the extinction--survival regimes
obtained in Theorem~\ref{thm:drift-weibull-frog} and
Corollary~\ref{cor:critical-drift-weibull}. The curve
\(\gamma_c(\beta)=1/\beta\) represents the critical relation
\(\beta=1/\gamma\). In the region \(\gamma<\gamma_c(\beta)\), the model
survives with positive probability in the direction of the drift. In the region
\(\gamma>\gamma_c(\beta)\), the model dies out almost surely whenever
\(\E(\eta)<\infty\). The boundary \(\gamma=\gamma_c(\beta)\) corresponds to the critical regime
covered by Corollary~\ref{cor:critical-drift-weibull}.

\begin{figure}[!ht]
  \centering
  \makebox[\textwidth][c]{%
    \includegraphics[width=1\linewidth,height=0.45\textheight,keepaspectratio=false]{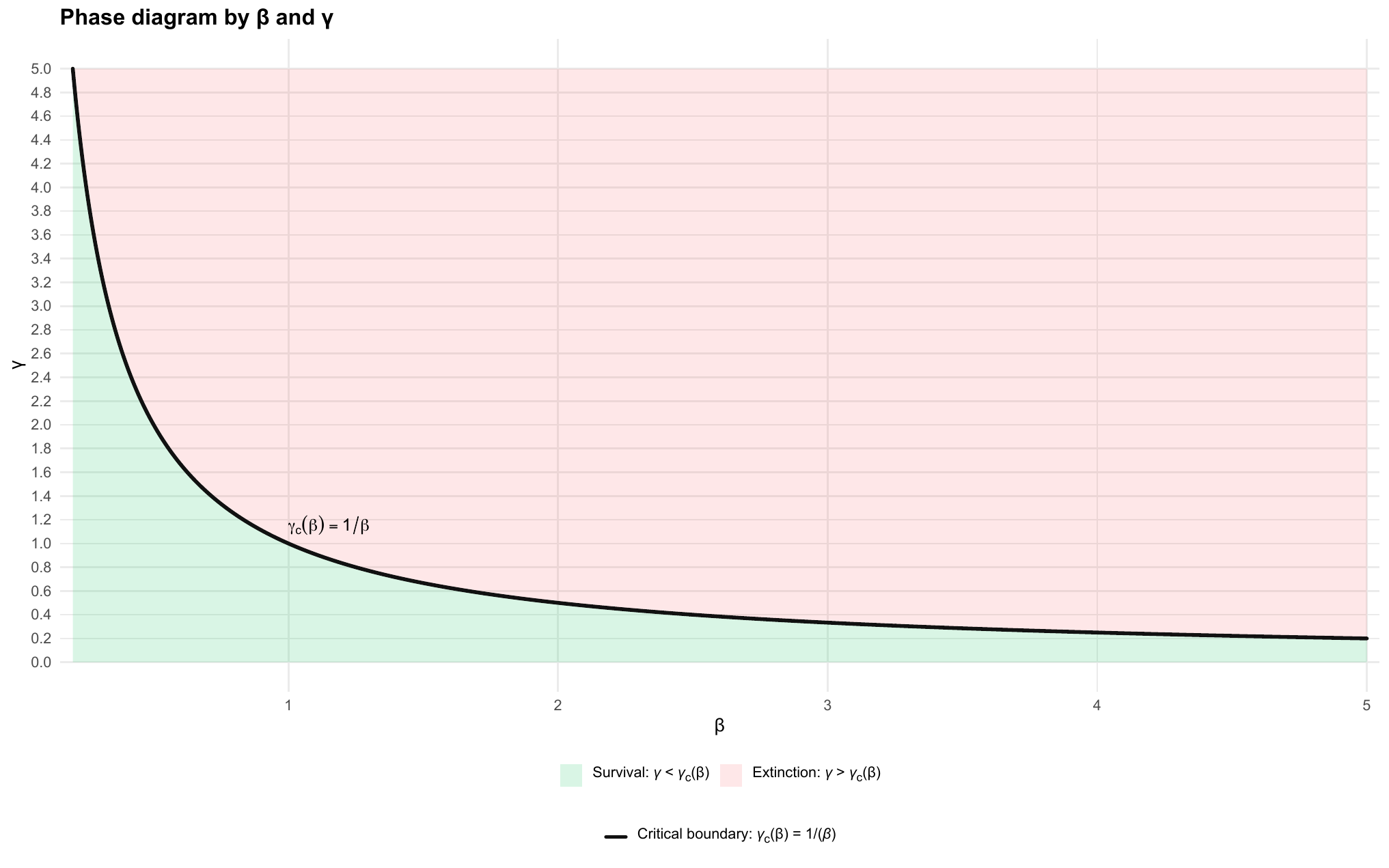}%
  }
  \caption{Phase diagram in the \((\beta,\gamma)\)-plane. The curve
\(\gamma_c(\beta)=1/\beta\) represents the critical boundary. The region
\(\gamma<\gamma_c(\beta)\) corresponds to survival with positive probability in
the direction of the drift, while the region \(\gamma>\gamma_c(\beta)\)
corresponds to almost sure extinction under the condition
\(\E(\eta)<\infty\).}
  \label{fig:diagrama_Weibull}
\end{figure}

\section*{Acknowledgements}
F.P.M. and J.H.R.G. are supported by FAPESP (grants 2023/13453-5 and 2025/03804-0, respectively).

\end{document}